\documentclass[12pt]{amsart}
\usepackage{amssymb, a4wide, mathdots,url,hyperref, shuffle}
\usepackage{bbm}
\usepackage{tikz}
\usepackage{ytableau}
\usepackage{pb-diagram}
\newcommand{\Ad}{\operatorname{Ad}}

\newcommand{\Lie}{\operatorname{Lie}}

\renewcommand{\subset}{\subseteq}

\newcommand{\Wh}{\operatorname{Wh}}

\newcommand{\mul}{\operatorname{Mult}}

\newtheorem{theorem}{Theorem}[section]
\newtheorem{lemma}[theorem]{Lemma}
\newtheorem{proposition}[theorem]{Proposition}
\newtheorem{remark}[theorem]{Remark}

\date{\today}

\DeclareMathOperator{\irr}{Irr}

\newcommand{\m}{\mathfrak{m}}

\renewcommand{\subset}{\subseteq}

\newcommand{\inte}{\mathrm{int}}

\begin{document}

\title[LCE for principal series]{On local character expansions for principal series representations of general linear groups}

%\begin{abstract}
%\end{abstract}

\begin{abstract}
We obtain two explicit formulas for the full local character expansion of any irreducible representation of a $p$-adic general linear group in principal blocks. The first, generalizing previous work of the author on the Iwahori-spherical case, expresses the expansion in terms of dimensions of degenerate Whittaker models. The second gives a closed expression in terms of values of Kazhdan-Lusztig polynomials of a suitable permutation group.

\end{abstract}

\author{Maxim Gurevich}
\address{Department of Mathematics, Technion -- Israel Institute of Technology, Haifa, Israel.}
\email{maxg@technion.ac.il}

\maketitle

\section{Introduction}
Two areas of study are typical focal points of character theory of representations of reductive $p$-adic groups. One examines how standard representations break down into irreducible subquotients. The second delves into asymptotic properties of irreducible representations, as encoded in microlocal invariants that are associated with adjoint nilpotent orbits of the group. The precise links between these two branches of study, particularly the role of the Kazhdan-Lusztig theory, which is known to govern the former problem, have often remained a subject of folklore.

In this note, we explicate two formulas that regard the Harish-Chandra--Howe local character expansion for principal series of general linear groups, by leveraging the well-developed Zelevinsky theory applicable to this context.

Let $G_n = GL_n(F)$ be the general linear locally compact group, defined over a $p$-adic field $F$. Let $\mathfrak{g}_n = \mathfrak{gl}_n(F)$ be its Lie algebra. The set $P(n)$, consisting of all integer partitions of $n$, naturally parameterizes the nilpotent $\Ad(G_n)$-orbits in $\mathfrak{g}_n$.

For a partition $\alpha\in P(n)$, the orbit $\mathcal{O}_\alpha$ is taken to be the set of nilpotent complex matrices, whose sizes of Jordan blocks are described by the parts of $\alpha$.

Given a finite-length smooth complex $G_n$-representation $\pi$, its trace is known to provide a distribution on the group, which is represented by a locally constant integrable function $\Theta_\pi$ on regular elements. For a regular element $X\in \mathfrak{g}_n$ close enough to zero, the celebrated Harish-Chandra--Howe local character expansion \cite{howe74,hcbook} takes the form
\begin{equation}\label{eq:expa}
\Theta_\pi(1+ X) = \sum_{\alpha\in P(n)} c_\alpha(\pi) \widehat{\mu}_{\mathcal{O}_\alpha}(X),\quad c_\alpha(\pi)\in \mathbb{Z}\;,
\end{equation}
where $\widehat{\mu}_{\mathcal{O}_\alpha},\, \alpha\in P(n)$, are suitably normalized functions representing the $\mathfrak{g}_n$-Fourier transform of the orbital integral distribution coming from the orbit $\mathcal{O}_\alpha$.

In \cite{me-triang}, with an assumption of a good residual characteristic for $F$, a formula was pointed out, which relates the set of integers $\{c_\alpha(\pi)\}_{\alpha\in P(n)}$, for an irreducible Iwahori-spherical $\pi$, with another well-studied set of integer invariants.

The second set is given by the dimensions of certain degenerate Whittaker models attached to the same $G_n$-representation $\pi$ and a nilpotent $\Ad(G_n)$-orbit.

% that is indexed by the same nilpotent orbits.

Namely, for $\alpha \in P(n)$, an exact normalized Jacquet functor $\mathbf{r}_{\alpha}$ is defined, which produces smooth $M_\alpha$-representations out of smooth $G_n$-representations, where $M_\alpha < G_n$ is a corresponding Levi subgroup.

%preserves finite-length.
We denote the integers
\[
d_\alpha(\pi) = \dim_{\mathbb{C}} \Wh(\mathbf{r}_{\alpha}(\pi)),\; \alpha\in P(n)\;,
\]
where $\Wh$ stands for the Whittaker functor on $M_\alpha$-representations.

Our first goal is to give a generalization of this formula that will encompass all representations in Bernstein blocks that arise from a minimal Levi subgroup (maximal torus), remove the residual characteristic constraint, and offer an alternative method of proof.

Let us make the stated notions more precise.

% repeat in a more precise manner the ingredients that were used for the statement of \cite[Theorem 1]{me-triang}.

For a partition $\alpha = (\alpha_1,\ldots,\alpha_k)$ in
\[
P(n) = \left\{ (\alpha_1 \geq \ldots \geq \alpha_k)\,:\, \alpha_i\in \mathbb{Z}_{>0},\; \alpha_1+\ldots + \alpha_k = n\right\}\;,
\]
the standard Levi subgroup $M_{\alpha} <G_n$ consists of block-diagonal matrices, with block sizes $\alpha_1,\ldots,\alpha_k$.

We also write $P_\alpha < G_n$ for the standard parabolic subgroup generated by $M_\alpha$ and the upper-triangular matrices in $G_n$, $N_\alpha< P_\alpha$ for its unipotent radical, and $\mathfrak{n}_\alpha = \Lie(N_\alpha) < \mathfrak{g}_n$.

Now, $\mathcal{O}_\alpha$ becomes the unique nilpotent $\Ad(G_n)$-orbit in $\mathfrak{g}_n$, for which $\mathfrak{n}_{\alpha^t} \cap \mathcal{O}_{\alpha}$ is dense in $\mathfrak{n}_{\alpha^t}$. Here, $\alpha \mapsto \alpha^t$ is the combinatorial transposition involution on $P(n)$.

%The spaces $\Wh(\mathbf{r}_{\alpha}(\pi))$, to which we will refer as principal degenerate Whittaker models, received attention as early as the foundational work of Zelevinsky \cite{Zel}. They were incorporated in the more general framework of degenerate Whittaker models of Moeglin--Waldspurger \cite{mw87}, and more recently were shown in \cite{ggs} to be minimal, in the proper sense, of such models.

%As customary,
We say that an irreducible $G_n$-representation $\pi$ is \textit{principal}, when $\mathbf{r}_{(1,1,\ldots,1)}(\pi)\neq 0$. In other words, there exists a character $\chi$ of the maximal torus $(F^\times)^n \cong T= M_{(1,\ldots,1)} < G_n$, for which $\pi$ is a sub-representation\footnote{Note, that in other common terminology a \textit{principal series} representation may refer to situations where $\pi= \mathbf{i}_{(1,\ldots,1)}(\chi)$ is irreducible. This notion is more restrictive than ours.} of $\mathbf{i}_{(1,\ldots,1)}(\chi)$, where $\mathbf{i}_{\alpha}$ is the normalized parabolic induction (adjoint to Jacquet) functor.

\begin{theorem}\label{thm:main}

For each principal irreducible representation $\pi$ of $G_n$ and each partition $\alpha\in P(n)$, the formula
\[
d_\alpha(\pi) = \sum_{\beta\in P(n)} s(\alpha,\beta^t) c_\beta(\pi)
\]
holds, where the coefficients $s(\alpha,\beta)\in \mathbb{Z}_{\geq 0}$ are the familiar combinatorial invariants given as

\begin{equation}\label{eq:graph}
s(\alpha,\beta) = \# \left\{ A \subset \{1,\ldots, k\} \times \{1,\ldots, l\} \;:\; \begin{array}{cc}
\alpha_ i = \# \{ j\;:\; (i,j)\in A\},\; \forall 1\leq i \leq k  \\
\beta_ j = \#\{ i\;:\; (i,j)\in A\},\; \forall 1\leq j \leq l  \end{array} \right\}\;,
\end{equation}

for $\alpha=(\alpha_1,\ldots,\alpha_k), \;\beta = (\beta_1,\ldots,\beta_l)$.

\end{theorem}

We invite the reader to the discussion in the introduction section of \cite{me-triang} on how the unitriangular (with respect to the topological closure order on nilpotent orbits) formula of Theorem \ref{thm:main} naturally refines the renowned result from \cite{mw87} concerning the relationship between character expansions and generalized Whittaker models.

Back in \cite{me-triang}, we utilized Hecke algebra techniques to reduce the problem, for the case of the Iwahori-spherical Bernstein block, to a matter involving the ring of symmetric functions. In Section \ref{sect:zel}, we instead highlight an alternative computation method for both sets of invariants, employing the basis of Zelevinsky-standard representations for the Grothendieck groups of $G_n$.

\subsection{Kazhdan-Lusztig theory}
The second goal of this note is to highlight a formula for producing the explicit values of the expansion $\{c_{\alpha}(\pi)\}_{\alpha\in P(n)}$, for a principal representation $\pi$, in terms of Kazhdan-Lusztig polynomials of a relevant permutations group.

We say that an irreducible $G_n$-representation $\pi$ is \textit{integral principal}, when it is principal and the irreducible subquotients of $\mathbf{r}_{(1,\ldots,1)}(\pi)$ are (unramified) characters $\chi$ of $(F^\times)^n$ of the form $\chi(a_1,\ldots, a_n) = \prod_{i=1}^n |a_i|_F^{t_i}$ with $t_i\in \mathbb{Z}$ ($|\,|_F$ being the $p$-adic norm).

We write $\irr_{\inte}(G_n)$ for the set of isomorphism classes of irreducible integral principal $G_n$-representations. As detailed at the head of Section \ref{sect:kl}, computations of local character expansions for principal representations are easily reduced to the integral case.

Let us now restate the Zelevinsky classification for $\irr_{\inte}(G_n)$.

For $k\geq1$, we consider the integer lattice $\mathbb{Z}^k$ with its cone $D_k= \{(\alpha_1 \geq \ldots \geq \alpha_k)\}\subset \mathbb{Z}^k$.

Considering the natural action of $S_k$ on $\mathbb{Z}^k$, the cone $D_k$ becomes a full set of orbit representatives. We denote the projection $p:\mathbb{Z}^k\to D_k$ that takes a tuple to its representative.

For $x_0\in \mathbb{Z}^k$, we would also like to consider a shifted $S_k$-action on $\mathbb{Z}^k$ given as
\[
\omega \ast_{x_0} x := \omega\cdot (x-x_0) + x_0\;,\quad \mbox{for}\;x\in \mathbb{Z}^k\;.
\]

We further denote the set
%We consider a setup of the form $P(n,k) \subset C(n,k) \subset \mathbb{Z}^k$, where
\[
C(n,k) = \{(\alpha_1,\ldots,\alpha_k)\;:\alpha_i\in \mathbb{Z}_{> 0}\;, \alpha_1 + \ldots + \alpha_k = n\}\subset \mathbb{Z}^k\;
\]
of compositions of $n$ into $k$ parts, the set of pairs
\[
M_{n,k} = \{(\lambda, \eta)\in \mathbb{Z}^k \times \mathbb{Z}^k\;:\; \eta-\lambda \in C(n,k)\}\;,
\]
and the set $\mul_{n,k} = M_{n,k}/S_k$ of orbits of the diagonal $S_k$-action on $M_{n,k}$.

An element of $\mul_{n,k}$, represented by a pair $((\lambda_1,\ldots,\lambda_k), (\eta_1,\ldots,\eta_k)) \in M_{n,k}$, can be thought of as a multiset of $k$ intervals $\{[\lambda_i,\eta_i]\}_{i=1}^k$ on the integer line, whose sum of lengths equals to $n$.

The Zelevinsky classification (\cite{Zel}) sets up a bijection
\begin{equation}\label{eq:zel}
Z: \bigcup_{k\geq 1} \mul_{n,k} \;\to\; \irr_{\inte}(G_n)\;.
\end{equation}

%Let us take note of several invariants produced by this bijection, for a representation $\pi = Z(\mathfrak{m})\in \irr_{\inte}(G_n)$.

%First, picking a representative $(\lambda,\eta) \in M_{n,k}$ for the orbit $\mathfrak{m}$, the partition $\mu(\mathfrak{m}):= p(\eta-\lambda)\in P(n)$ is clearly well-defined.

%Indeed, by works of Zelevinsky .. and M-W .., this is the celebrated wavefront set invariant. Namely, an orbit $\mathcal{O}_{\mu}$ is contained in the topological closure of $\mathcal{O}_{\mu(\mathfrak{m})}$, for any partition $\mu\in P(n)$ with a non-vanishing coefficient $c_{\mu}(\pi)$ (or $d_{\mu}(\pi)$).

%For a representation $\pi = Z(\mathfrak{m})\in \irr_{\inte}(G_n)$, we set $\omega(\pi)=\omega(\mathfrak{m})\in S_k$ to be the longest permutation that can be written as $\omega(\pi)= \omega_2\omega_1^{-1}$, for permutations that satisfy $\omega_1\cdot \lambda, \omega_2\cdot \eta\in D_k$, while $(\lambda,\eta)\in M_{n,k}$ is any pair belonging to the $S_k$-orbit $\mathfrak{m}$.

%Further considering the sets
%\[
%C^0(n,k) = \{(\alpha_1,\ldots,\alpha_k)\;:\alpha_i\in \mathbb{Z}_{ \geq 0}\;, \alpha_1 + \ldots + \alpha_k = n\}\;,\quad P^0(n,k) = C^0(n,k) \cap D_k\;,
%\]
%of compositions and partitions of $n$ into at most $k$ parts, we see that $p$ restricts to the natural projection $p: C^0(n,k) \to P^0(n,k)$.

%For $\pi = Z(\lambda , \alpha)\in \irr_{\inte}(G_n)$, we write $\omega(\pi) = \omega(\lambda,\alpha) \in S_k$ for the longest permutation satisfying
%\[
%\omega(\pi)\cdot(\lambda+\alpha) \in D_k\;.
%\]

For a partition $\mu = (\mu_1,\ldots,\mu_t)\in P(n)$ with $t\leq k$, we consider $\mu$ as an element of $D_k$ by padding its tail with $0$'s if necessary.

\begin{theorem}\label{thm:kl}
Let $\pi = Z(\mathfrak{m})\in \irr_{\inte}(G_n)$ be an irreducible integral principal representation, parameterized by the $S_k$-orbit $\mathfrak{m}$ of a pair $(\lambda,\eta)\in M_{n,k}$.

Let $\alpha = \eta-\lambda\in C(n,k)$ be the associated composition. Let $\omega_1,\omega_2\in S_k$ be the permutations for which the product $\omega_1\omega_2$ is longest possible while subject to
\[
\omega_1\cdot \lambda, \,\omega_2^{-1}\cdot\eta \in D_k\;.
\]

Then, for a partition $\mu\in P(n)$ with $\mu_1\leq k$, the identity
\[
c_{\mu}(\pi) = \sum_{ \tau\in S_k\;:\; p(\tau\ast_{-\lambda} \alpha) = \mu^t} \epsilon(\tau) P_{\omega_0 \omega_1\tau\omega_2  ,\, \omega_0 \omega_1\omega_2}(1)
\]
holds, where $P_{\zeta_1,\zeta_2}$ stands for the Kazhdan-Lusztig polynomial attached to a pair of permutations $\zeta_1,\zeta_2\in S_k$, $\epsilon:S_k \to\{\pm1\}$ is the sign map, and $\omega_0\in S_k$ denotes the longest permutation in the group.

For partitions $\mu\in P(n)$ with $\mu_1> k$, we have $c_{\mu}(\pi) =0$.

\end{theorem}

Given that Bernstein blocks of $G_n$-representations are inherently equivalent to module categories over type $A$ affine Hecke algebras, it is reasonable to speculate that the formula presented in Theorem \ref{thm:kl} is expressible using Kazhdan-Lusztig polynomials associated with the affine Weyl group of type $A$. However, the author has not encountered a documentation of this particular approach in literature. We suggest further investigation into the role of Kazhdan-Lusztig theory in the context of local character expansion coefficients.

%Let us also comment on that formula from the point of view of the Langlands reciprocity. While the Zelevinsky classification should certainly be viewed as an early instance of the reciprocity for general linear groups, this computational aspects come with further complications. Namely, the Langlands parameter that corresponds to a given $\pi\in \irr_{\int}(G_n)$ is naturally identifiable with a multisegment $\mathfrak{m_\pi}\in \mul_{k,n}$, through a correspondence of the form $\pi = Z(\widehat{\mathfrak{m}_{\pi}})$, where $\mathfrak{m} \mapsto \widehat{\mathfrak{m}}$ is the Zelevinsky involution $\bigcup_{k\geq 1} \mul_{n,k}$, which was algorithmically explicated in \cite{mw-zel}.

%Thus, in complexity terms, the formula of Theorem \ref{thm:kl} could give the local character expansion of a principal series representations, when fed with algorithms for $S_n$-Kazhdan-Lusztig polynomials supplemented with an algorithm for the Zelevinsky involution. Indeed, this comes in line with the recent insights given in .., where the wavefront set invariant for unipotent representations of further $p$-adic reductive groups was shown to be directly extracted out of the Aubert-Zelevinsky transform of its Langlands parameter.

We should also comment on the formula from the perspective of the Langlands reciprocity. While it is clear that the Zelevinsky classification is an early instance of the reciprocity for general linear groups, it involves additional computational issues. Specifically, the Langlands parameter corresponding to a given representation $\pi \in \irr_{\inte}(G_n)$ can be naturally associated with a multisegment $\mathfrak{m}_\pi \in \mul_{k,n}$ through a correspondence of the form $\pi = Z(\widehat{\mathfrak{m}_{\pi}})$, where the mapping $\mathfrak{m} \mapsto \widehat{\mathfrak{m}}$ is the Zelevinsky involution on $\bigcup_{k\geq 1} \mul_{n,k}$, which was algorithmically explicated in \cite{mw-zel}.

Thus, in terms of complexity, the formula of Theorem \ref{thm:kl} yields the local character expansion of principal representations when provided with algorithms for $S_n$-Kazhdan-Lusztig polynomials, complemented by an algorithm for the Zelevinsky involution. This observation aligns with recent insights presented in \cite{ciub21,ciub23}, where it was demonstrated that the wavefront set invariant for further cases of unipotent representations of $p$-adic reductive groups can be directly extracted from the Aubert-Zelevinsky transform of its Langlands parameter.

\subsection{Acknowledgements}

The note has been shaped by insights gained from engaging discussions I had during participation in the July 2022 IMS workshop. I am particularly grateful to Rapha\"{e}l Beuzart-Plessis, Monica Nevins, Emile Okada, and Marie-France Vign\'{e}ras for sharing their perspectives on this matter. I extend my appreciation to the organizers of this wonderful workshop in Singapore.

This research is supported by the Israel Science Foundation (Grant Number: 737/20).

\section{Zelevinsky-standard representations}\label{sect:zel}

Let $\widehat{F^\times}$ be the collection of smooth complex one-dimensional $GL_1(F)$-representations (characters). For $\chi\in \hat{F^\times}$, let $\Re(\chi)\in \mathbb{R}$ be the number for which $|\chi(a)| = |a|_F^{\Re(\chi)}$ holds, for all $a\in F^\times$.

One-dimensional $G_n$-representations are of the form $\chi^{(n)}:= \chi\circ \det$, for $\chi\in \widehat{F^\times}$.

Given a composition $\alpha= (\alpha_1,\ldots,\alpha_k)\in C(n,k)$, in similarity with the partition case, we let $P_\alpha = M_\alpha N_\alpha < G_n$ be the corresponding standard parabolic subgroup. As customary, we treat the corresponding normalized parabolic induction functor $\mathbf{i}_{\alpha}$ as a product functor
\[
\pi_1\times\cdots\times \pi_k = \mathbf{i}_{\alpha}(\pi_1\boxtimes\cdots \boxtimes \pi_k)
\]
that takes $G_{\alpha_1}\times\cdots \times G_{\alpha_k}$-representations to $G_n$-representations.

For forthcoming needs, we would like to consider the larger set
\[
C_0(n,k) = \{(\alpha_1,\ldots,\alpha_k)\;:\alpha_i\in \mathbb{Z}_{\geq 0}\;, \alpha_1 + \ldots + \alpha_k = n\}\subset \mathbb{Z}^k\;
\]
of compositions with possible zero parts.
The definitions of the functors $\mathbf{r}_{\alpha}$ and $\mathbf{i}_{\alpha}$ are extended to $\alpha \in C_0(n,k)$ by considering $G_0$ as a trivial group and its representations as neutral to the parabolic induction product. In particular, for $\chi\in \widehat{F^\times}$, we may take $\chi^{(0)}$ as a trivial representation of $G_0$.

For $\alpha\in C_0(n,k)$ and characters $\chi_1,\ldots,\chi_k\in \widehat{F^\times}$, a $G_n$-representation
\begin{equation}\label{eq:zeta1}
\zeta = \zeta(\alpha,\chi_1,\ldots,\chi_k):= \chi_1^{(\alpha_1)}\times \cdots \times \chi_k^{(\alpha_k)}
\end{equation}
is constructed. We say that $\zeta$ is a \textit{principal Zelevinsky-standard representation} whenever $\Re(\chi_i) \geq \Re(\chi_j)$ holds, for all $1\leq i< j \leq k$.

In \cite{Zel} it is shown that each irreducible $G_n$-representation $\pi$ corresponds to a unique Zelevinsky-standard representation $\zeta(\pi)$, in which $\pi$ appears as a sub-representation. It is also a standard corollary that the collection of Zelevinsky-standard representations forms a basis for the Gorthendieck group of the category of finite-length $G_n$-representations.

Finally, it is a standard outcome of the theory of Bernstein blocks that the collection of principal Zelevinsky-standard representations forms a basis for the Gorthendieck group of the direct summand of the category that is generated by the principal irreducible representations.

Note now, that due to trace additivity and exactness of Whittaker and Jacquet functors, both $c_\mu(\cdot)$ and $d_\mu(\cdot)$, for $\mu\in P(n)$, are additive functionals on the above mentioned Grothendieck groups.

This short discussion sums up to the following reduction.

\begin{proposition}
Theorem \ref{thm:main} will follow, if the identity in its statement is proved with $\pi$ being a principal Zelevinsky-standard representation, rather than an irreducible representation.
\end{proposition}

%We denote the set of \textit{character-decorated partitions}
%\[
%\mathcal{Z}(n)= \{(\alpha, \chi_1,\ldots, \chi_{\ell(\alpha)})\;:\; \alpha\in P(n),\; \chi_i\in \hat{F^\times}\; \forall i\}\;.
%\]

%For $\underline{\alpha} = (\alpha,\chi_1,\ldots, \chi_{\ell(\alpha)})\in \mathcal{Z}(n)$, a \textit{Zelevinski-standard} $G_n$-representation

%\[
%\zeta(\underline{\alpha}):= \chi_{\omega(1)}^{ ( \alpha_{\omega(1)} ) } \times \cdots \times \chi_{\omega(\ell(\alpha) )}^{ ( \alpha_{\omega(\ell(\alpha))} ) }
%\]
%is attached, by taking any permutation $\omega\in S_{\ell(\alpha)}$, which satisfies

%The isomorphism class of $\zeta(\underline{\alpha})$ does not depend on the choice of such $\omega\in S_{\ell(\alpha)}$.

%\begin{proposition}[Zelevinsky classification for principal series]

%There is a bijection $\underline{\alpha}\mapsto Z(\underline{\alpha})$, between $\mathcal{Z}(n)$ and the set of isomorphism classes of principle series $G_n$-representation, constructed by taking %$Z(\underline{\alpha})$ as the unique irreducible subrepresentation of the Zelevinsky-standard representation $\zeta(\underline{\alpha})$.

%\end{proposition}

Indeed, Theorem \ref{thm:main} now clearly follows from the previous proposition and the following two lemmas.

\begin{lemma}\label{lem-2}

For a principal Zelevinsky-standard representation $\zeta =\zeta(\widetilde{\beta} ,\chi_1,\ldots,\chi_k)$ with $\widetilde{\beta}\in C(n,k)$, and a given partition $\alpha\in P(n)$, we have\footnote{Note that the application of $p$ in the formula is somewhat redundant, since $s(\alpha,\widetilde{\beta})$ may be defined directly for compositions that are not necessarily partitions.}
\[
d_{\alpha}(\zeta) = s(\alpha, p(\widetilde{\beta}))\;.
\]
\end{lemma}

\begin{proof}

We apply the basic Mackey theory for $p$-adic general linear groups, as stated for example in  \cite[2.2]{LM2} (a case of the Bernstein-Zelevinsky geometric lemma). It follows that the $M_\alpha= M_{(\alpha_1,\ldots,\alpha_l)}$-representation $\mathbf{r}_{\alpha}(\zeta)$ is decomposed into sub-quotients of the form
\[
(\chi_{1,1}\times \cdots \times\chi_{k,1})\boxtimes\cdots \boxtimes (\chi_{1,l}\times \cdots \times \chi_{k,l})\;,
\]
for all possible $k\times l$ matrices $A = (\chi_{i,j})$ of representations, such that $\chi_{i,1} \boxtimes \cdots \boxtimes  \chi_{i,l}$ is an irreducible sub-quotient of $\mathbf{r}_{\gamma_{i}}(\chi_i^{(\beta_i)})$, for all $i=1,\ldots,k$, and choices of $\gamma_i\in C_0(\beta_i,l)$ for which the group ranks fit.

Let $\mathcal{A}$ denote the set of such matrices $A$, up to isomorphism of their entries.

%In this notation we allow $\gamma_{i}$ to have zero parts, in which case the corresponding $\chi_{i,j}$ is viewed as a representation of $G_0$ neutral to the parabolic induction product.

By exactness properties of the Whittaker functor, $d_\alpha(\zeta)$ will count the cardinality of the subset $\mathcal{A}_{\mathrm{gen}}\subset \mathcal{A}$, consisting of matrices all of whose entries are generic representations (assuming $G_0$-representations are counted as generic).

Now, since all $\chi_i^{(\beta_i)}$ are one-dimensional, so are all possible entries of matrices in $\mathcal{A}$. Since one-dimensional $G_m$-representations are generic, if and only if, $m\leq 1$, the set $\mathcal{A}_{\mathrm{gen}}$ is identified with the set of $k\times l$ matrices with entries in $\{0,1\}$, for which rows sum up to the values of $\beta_i$ while columns sum up to the values of $\alpha_j$. The cardinality of this set is given by the number $s(p(\widetilde{\beta}),\alpha)= s(\alpha,p(\widetilde{\beta}))$.

\end{proof}

\begin{lemma}\label{lem-1}
For a principal Zelevinsky-standard representation $\zeta =\zeta(\alpha ,\chi_1,\ldots,\chi_k)$ with $\alpha\in C(n,k)$, and a given partition $\beta\in P(n)$, we have
\[
c_{\beta}(\zeta) = \left\{ \begin{array}{ll}  1 &  p(\alpha) = \beta^t  \\ 0 & p(\alpha) \neq \beta^t  \end{array}  \right.\;.
\]

\end{lemma}

The proof of Lemma \ref{lem-1} will follow readily from the multiplicative nature of local character expansions, as was recently explicated by Henniart-Vign\'{e}ras. We recall it here in a form convenient for our needs.

\begin{proposition}[Theorem 1.5 \cite{henvig}]\label{prop:h-v}\footnote{  In \cite{henvig}, notation for $c_\mu$ corresponds to our $c_{\mu^t}$, which results in a visible difference between sources in descriptions of the multiplicative property.}
Given a tuple of finite-length $G_{\alpha_i}$-representations $\pi_i$, $i=1,\ldots,k$, with $(\alpha_1,\ldots,\alpha_k)\in C(n,k)$, and a partition $\mu\in P(n)$, we have
\[
c_{\mu}(\pi_1\times\cdots \times \pi_k) = \sum c_{\beta^1}(\pi_1)\cdots c_{\beta^k}(\pi_k)\;,
\]
where the sum is over all tuples of partitions $\beta^i \in P(\alpha_i)$, $i=1,\ldots,k$, that satisfy
\[
p(\beta^1 + \ldots + \beta^k) = \mu\;.
\]
Here, partitions are considered as elements of the lattice $\mathbb{Z}^n$ through $0$ padding as before.
\end{proposition}

\begin{remark}
We have until now left the coefficients $c_{\mu}(\cdot)$ defined only up to a constant factor, without specifying a particular normalization. Indeed, Proposition \ref{prop:h-v} remains valid when a specific normalization is imposed, as outlined in \cite{henvig}. However, instead of delving into the details of this normalization, which would require us to deviate somewhat from the main discussion, we opt to interpret Lemma \ref{lem-1} as a convenient way to pin down the constants.

In essence, it suffices to avoid fixing a normalization while regarding the lemma as a statement on vanishing behavior of coefficients, under the assumption that Proposition \ref{prop:h-v} is valid.
\end{remark}

\begin{proof}[Proof of Lemma \ref{lem-1}]

Since all $\chi_i^{(\alpha_i)}$ are one-dimensional representations, the only contribution for their local character expansions comes from the zero orbit. In other words, we have $c_{\beta}( \chi_i^{(\alpha_i)}) \neq 0$, if and only if, $\beta= \kappa(\alpha_i)$, where $\kappa(r) = (1,\ldots,1) \in P(r)$ stands for the minimal partition of an integer $r\geq1$.

%\[
%c_{\beta}( \chi_i^{(\alpha_i)}) = \left\{\begin{array}{ll}  1  & \beta = \kappa(\alpha_i) \\ 0  & \mbox{otherwise}  \end{array} \right.\;,
%\]

By Proposition \ref{prop:h-v}, it now follows that the only non-vanishing local character expansion coefficient of the Zelevinsky-standard representation $\zeta$ is parameterized by the partition
\[
\mu = p(\kappa(\alpha_1) + \ldots + \kappa(\alpha_k))\in P(n)\;.
\]
It is easy to verify that $\mu = \alpha^t$.

\end{proof}

\section{Explicit expansions}\label{sect:kl}

%We say that $\zeta$ is an \textit{integral Zelevinsky-standard representation}, if it is a principal Zelevinsky-standard representation of the form $\zeta = \zeta(\alpha, \nu^{t_1}, \ldots, \nu^{t_k})$.

%An irreducible $G_n$-representation $\pi$ belongs to $\irr_{\inte}(G_n)$, if and only if, $\zeta(\pi)$ is integral.

The theory of \cite{Zel} stipulates that any principal irreducible $G_n$-representation $\pi$ can be constructed in the form
\[
\pi = (\chi_1^{(\gamma_1)}\otimes \pi_1) \times\cdots \times (\chi_r^{(\gamma_r)}\otimes \pi_r)\;,
\]
for $\pi_i \in \irr_{\inte}(G_{\gamma_i})$, $i=1,\ldots,r$, and characters $\chi_1,\ldots,\chi_r\in \widehat{F^\times}$.

Since tensoring with one-dimensional representations does not affect local character expansions, it follows from the above decomposition and Proposition \ref{prop:h-v} that the formula of Theorem \ref{thm:kl} suffices to give an explicit local character expansions for all principal irreducible $G_n$-representations, in terms of their Zelevinsky parameters.

%Let us now recall the Zelevinsky classification of $\irr_{\inte}(G_n)$ in a more familiar notation of multisegments, though still adapted to our current tools.
Let us now recall the details of the construction of the Zelevinsky bijection \eqref{eq:zel}.

For $t\in \mathbb{Z}$, we write $\nu^t \in \widehat{F^\times}$ for the unramified character given by $\nu^t(a) = |a|^t_F$.

We consider the set of pairs
\[
\overline{M}_{n,k} = \{(\lambda, \eta)\in D_k \times \mathbb{Z}^k\;:\; \eta-\lambda \in C_0(n,k)\}\;.
\]
For a pair $(\lambda,\eta)\in \overline{M}_{n,k}$, a $G_n$-representation is constructed by parabolic induction of the form
\begin{equation}\label{eq:zeta2}
\zeta(\lambda, \eta) = \zeta( \eta- \lambda, \nu^{\frac12(\lambda_1 + \eta_1+1)}, \ldots, \nu^{\frac12(\lambda_k + \eta_k+1)})\;.
\end{equation}

We will say that such $\zeta(\lambda, \eta)$ is an \textit{integral Zelevinsky-standard representation}.

%Note, that by scrapping zero parts out of the composition, it is evident that for each $(\lambda,\eta)\in \overline{M}_{n,k}$, there is a unique $k'\leq k$ and a unique pair $(\lambda
For an orbit $\mathfrak{m}\in \mul_{n,k}$, it is evident that there are unique tuples $\lambda(\mathfrak{m}), \eta(\mathfrak{m})\in D_k$, for which
\[
\left(\lambda(\mathfrak{m}),\omega\cdot \eta(\mathfrak{m})\right)\in \overline{M}_{n,k}
\]
is satisfied, for a permutation $\omega\in S_k$. We set $\omega = \omega(\mathfrak{m})\in S_k$ to be the longest such permutation.

We set
\[
\zeta\left(\mathfrak{m}) = \zeta(\lambda(\mathfrak{m}),\omega(\mathfrak{m})\cdot \eta(\mathfrak{m})\right)\;.
\]

The bijection \eqref{eq:zel} is now constructed by taking $Z(\mathfrak{m})$ to be the unique irreducible sub-representation of $\zeta(\mathfrak{m})$, for each $\mathfrak{m}\in  \mul_{n,k}$.

%A specialization of the Zelevinsky classification may now be stated as follows.
%\begin{proposition}
%A bijection $Z:\cup_{k\geq1} \mul_{n,k} \to \irr_{\inte}(G_n)$ exists, so that $\zeta(\mathfrak{m})$ and $\zeta(Z(\mathfrak{m}))$ are isomorphic representations, for all $\mathfrak{m}\in \mul_{n,k}$.

%In particular, each $\pi \in \irr_{\inte}(G_n)$ is constructed as the unique irreducible sub-representation of $\zeta(\mathfrak{m})$, for a unique multisegment $\mathfrak{m}\in \cup_{k\geq1} \mul_{n,k}$.
%\end{proposition}

\begin{remark}
Each integral Zelevinsky-standard representation $\zeta(\m)$ is isomorphic to a principal Zelevinsky-standard representation.

This fact is not tautological, since the condition in \eqref{eq:zeta1} requires $\frac12(\lambda(\m) + \omega(\m)\eta(\m)) \in D_k$ to hold, while the construction of $\zeta(\m)$ implies only the non-equivalent condition $\lambda(\m) \in D_k$. Yet, basic commutation relations on the parabolic induction product provide this fact.

%Notably,
% for a given pair $(\lambda,\eta)\in M_{n,k}$, the conditions of $\lambda\in D_k$ and $\frac12(\lambda + \eta) \in D_k$ are \textit{not} equivalent, and cause a technical discrepancy between the conditions we required for the constructions in \eqref{eq:zeta1} and \eqref{eq:zeta2} to produce a Zelevinsky-standard representation. Yet, it is always the case that
\end{remark}

We now recall the role of Kazhdan-Lusztig polynomials in the character theory of finite-length $G_n$-representations. This perspective was initially shaped by the conjectures of \cite{zel-kl}, which were subsequently proved in the affine Hecke algebra setting in \cite{ginz-book}, by means of geometric realization of standard modules. An account of this theory can be found in the survey \cite[2.3]{LNT}, while a precise translation of these findings into the $p$-adic group context can be read in \cite[Section 3]{me-restriction}. We also recommend the introduction section of \cite{Hender} for an effective overview of this theory.

An alternative approach may employ the KLR categorification of the dual canonical basis for quantum groups, coupled with Lusztig's description of the transition matrix between PBW and canonical bases and categorical equivalences with the representation theory of affine Hecke algebras. For explorations of this approach, the reader may delve into the details presented in \cite[Section 3]{gur-klrrsk} or \cite{me-parabkl}.

With this background in hand, we limit ourselves to a presentation of an established corollary of the aforementioned discussion.

For a finite-length $G_n$-representations $\pi$, let $[\pi]$ denote its projection into the Grothendieck group of $G_n$.

\begin{proposition}\label{prop:klcoeff}
For a multisegment $\mathfrak{m}\in \mul_{n,k}$, we have an expansion
\[
\left[Z(\mathfrak{m})\right]=  \sum_{\sigma\in S_k\;:\; (\lambda(\mathfrak{m}),\,\sigma\cdot \eta(\mathfrak{m}))\in \overline{M}_{n,k}} \epsilon(\sigma \omega) P_{\omega_0\sigma, \omega_0\omega(\mathfrak{m})} (1) \left[\zeta(\lambda(\mathfrak{m}),\,\sigma\cdot \eta(\mathfrak{m}))\right]\;,
\]
of the corresponding irreducible $G_n$-representation on the basis of Zelevinsky-standard representations.

\end{proposition}

\begin{proof}
  The formula is given explicitly, for example, in \cite[Corollary 10.1]{LM3}.
\end{proof}

\begin{proof}[Proof of Theorem \ref{thm:kl}]
Note, that $\omega(\mathfrak{m}) = \omega_1\omega_2$, $\omega_1 \cdot \lambda = \lambda(\mathfrak{m})$ and $\eta = \omega_2\cdot \eta(\mathfrak{m})$ hold.

Let us set the composition
\[
\alpha(\mathfrak{m}) = \omega(\mathfrak{m})\cdot \eta(\mathfrak{m}) - \lambda(\mathfrak{m}) = \omega_1\cdot \alpha\in  C(n,k)\;.
\]
Since $C_0(n,k)$ is a $S_k$-invariant set, the equivalence
\[
\tau \ast_{-\lambda} \alpha \in C_0(n,k)\; \Leftrightarrow\; (\omega_1\tau\omega_1^{-1})\ast_{-\lambda(\mathfrak{m})} \alpha(\mathfrak{m})\in C_0(n,k)
\]
holds, for all $\tau\in S_k$, with both compositions belonging to the same $S_k$-orbit.

Hence, it suffices to prove the identity
\[
c_{\mu}(\pi) = \sum_{ \tau\in S_k\;:\; p(\tau\ast_{-\lambda(\mathfrak{m})} \alpha(\mathfrak{m})) = \mu^t} \epsilon(\tau) P_{\omega_0\tau \omega(\mathfrak{m}) ,\, \omega_0 \omega(\mathfrak{m})}(1)\;.
\]

Clearly, for $\tau\in S_k$, the conditions $\tau \ast_{-\lambda(\mathfrak{m})} \alpha(\mathfrak{m}) \in C_0(n,k)$ and $(\lambda(\mathfrak{m}),\,\tau\omega(\mathfrak{m})\cdot \eta(\mathfrak{m}))\in \overline{M}_{n,k}$ are equivalent.

Moreover, by Lemma \ref{lem-1}, for $\tau\in S_k$ that satisfies the above condition, we have
\[
c_{\mu}(  \zeta(\lambda(\mathfrak{m}),\,\tau\omega(\mathfrak{m})\cdot \eta(\mathfrak{m}))  ) =  \left\{ \begin{array}{ll}  1 &  p(\tau \ast_{-\lambda(\mathfrak{m})} \alpha(\mathfrak{m})) = \mu^t  \\ 0 & \mbox{otherwise}  \end{array}  \right.\;.
\]

The identity now follows by exactness from Proposition \ref{prop:klcoeff}.

\end{proof}

\bibliographystyle{alpha}
\bibliography{propo2}{}

\end{document}